\documentclass{amsart}
\makeatletter

\def\ps@plain{%
  \ps@empty
  \def\@oddfoot{%
    \normalfont\large
    \hfil\thepage\hfil
  }%
  \let\@evenfoot\@oddfoot
}

\makeatother

\usepackage{verbatim}
\usepackage{xcolor}
\usepackage{tikz}
\usepackage{tikz-cd}
\usetikzlibrary{arrows.meta}
\usepackage{arydshln}
\usepackage{caption}
\usepackage{subcaption}
\usepackage{graphicx} 
\usepackage[hyphens]{url}
\usepackage{amsmath}
\usepackage{bbm}
\usepackage{hyperref}
\usepackage{mathrsfs}
\usepackage{todonotes}
\usepackage{amssymb}
\usepackage{mathabx}
\usepackage{enumitem}

\newtheorem{theorem}{Theorem}[section]
\newtheorem{proposition}[theorem]{Proposition}
\newtheorem{lemma}[theorem]{Lemma}
\newtheorem{corollary}[theorem]{Corollary}

\theoremstyle{definition}
\newtheorem{definition}[theorem]{Definition}

\theoremstyle{remark}
\newtheorem{remark}[theorem]{Remark}

\numberwithin{equation}{section}

\title{Totally Geodesic Submanifolds of Teichm\"uller Space in Complex Dimension at Least Three}
\author{Jia Longsong}
\address{School of Mathematical Sciences, Peking University, Beijing, 100871, P. R. China.}
\email{jialongsong@stu.pku.edu.cn}
\begin{document}
\maketitle

\begin{abstract}
Let $N\subset \mathcal{T}_{g,n}$ be a connected complex totally geodesic submanifold. We prove that if $\dim_{\mathbb{C}}N\ge3$, then $N$ is the image of an entire Teichm\"uller space under a covering construction induced by a finite branched cover of marked surfaces. This answers a question of Arana-Herrera and Wright.
\end{abstract}
\section{Introduction}

Let $\mathcal{T}_{g,n}$ denote the Teichm\"uller space of Riemann surfaces of genus $g$ and $n$ marked points, endowed with the Teichm\"uller metric. A connected complex submanifold $N\subset \mathcal{T}_{g,n}$ is \emph{totally geodesic} if the Teichm\"uller disc determined by any two distinct points of $N$ is contained in $N$.

In complex dimension one, totally geodesic submanifolds are precisely Teichm\"uller discs. In higher dimensions totally geodesic submanifolds are substantially more rigid. Wright proved that every totally geodesic submanifold of complex dimension greater than one projects to a closed algebraic totally geodesic subvariety of $\mathcal{M}_{g,n}$, and for fixed $(g,n)$ there are only finitely many such subvarieties \cite{MR4120819}.

The basic construction of higher-dimensional examples is provided by \emph{covering constructions}. A finite branched cover of marked topological surfaces $h:(S_X,D)\to (S_Y,B)$ with the marked points compatible with the ramification induces by pullback of complex structures a holomorphic isometric embedding
\begin{equation} f_h:\mathcal{T}(S_Y,B)\to \mathcal{T}(S_X,D). \nonumber \end{equation}

Arana-Herrera and Wright asked whether every totally geodesic submanifold of complex dimension at least three is a covering construction based on an entire Teichm\"uller space \cite[Question~10.1]{AranaHerreraWright2025}. Our main theorem gives an affirmative answer.

\begin{theorem}\label{thm:main}
Let $(S_X,D)$ be the marked topological surface underlying $\mathcal{T}_{g,n}$, and let $N\subset \mathcal{T}_{g,n}$ be a connected complex totally geodesic submanifold. If $\dim_{\mathbb{C}}N\ge 3$, then there exist a marked surface $(S_Y,B)$ and a finite orientation-preserving topological branched cover
\begin{equation} h:S_X\to S_Y \nonumber \end{equation}
such that
\begin{equation} h^{-1}(B)=D\cup\operatorname{Ram}(h) \nonumber \end{equation}
and
\begin{equation} N=f_h\bigl(\mathcal{T}(S_Y,B)\bigr). \nonumber \end{equation}
\end{theorem}

\medskip

There is a useful consequence for Teichm\"uller dynamics. Let
\begin{equation} \pi:\mathcal{Q}(\kappa)\to \mathcal{M}_{g,n}, \qquad (X,q)\longmapsto X, \nonumber \end{equation}
denote the projection from a principal stratum of quadratic differentials. Kahn and Wright observed that the projection of any $GL_2(\mathbb{R})$-orbit closure in the principal stratum is totally geodesic; moreover, the dimension of the projection is one half of the dimension of the orbit closure \cite[Proposition~1.5]{MR4447597}. Combining this with Theorem~\ref{thm:main} gives the following.

\begin{corollary}\label{cor:orbit-closures}
Let $\mathcal{M}$ be a $GL_2(\mathbb{R})$-orbit closure in the principal stratum of quadratic differentials over $\mathcal{M}_{g,n}$. If $\dim_{\mathbb{C}}\pi(\mathcal{M})\ge 3$, equivalently, if $\dim_{\mathbb{C}}\mathcal{M}\ge 6$, then every connected lift of $\pi(\mathcal{M})$ to $\mathcal{T}_{g,n}$ is a covering construction.
\end{corollary}

\begin{proof}
By \cite[Proposition~1.5]{MR4447597}, the subvariety $\pi(\mathcal{M})$ is totally geodesic and
\begin{equation} \dim_{\mathbb{C}}\pi(\mathcal{M})=\frac{1}{2}\dim_{\mathbb{C}}\mathcal{M}. \nonumber \end{equation}
The conclusion therefore follows from Theorem~\ref{thm:main}.
\end{proof}

In particular, among totally geodesic subvarieties obtained as projections of orbit closures in the principal stratum, any primitive example of dimension greater than one can occur only in complex dimension two. This shows that the known primitive surfaces of \cite{MR3664815,MR4155219} occur precisely in the only dimension in which primitive higher-dimensional examples can exist.

The result most closely related to Theorem~\ref{thm:main} is the rigidity theorem of Benirschke and Serv\'an, which shows that holomorphic isometric embeddings between Teichm\"uller spaces are induced by covering constructions \cite{MR4768147}. Theorem~\ref{thm:main} does not follow directly from the rigidity theorem for isometric embeddings of Teichm\"uller spaces. In that setting both Teichm\"uller spaces are given in advance, and hence both cotangent spaces are already full spaces of quadratic differentials on known Riemann surfaces. For a general totally geodesic submanifold $N$, there is no Riemann surface downstairs to start with.

Fix $X\in N$, and let $D_X$ denote its marked points. Let $Q_X N\subset Q(X,D_X)$ be the space of quadratic differentials whose Teichm\"uller discs are contained in $N$. By \cite{MR4768147}, these spaces form a holomorphic vector bundle over $N$ of rank $d=\dim_{\mathbb{C}}N$. Cotangent restriction identifies $Q_X N$ isometrically with $T_X^*N$, and hence determines a norm-one projection
\begin{equation} P_X:Q(X,D_X)\to Q_X N. \nonumber \end{equation}

The central problem of the paper is therefore a reconstruction problem: starting only from the intrinsic subspace $Q_X N\subset Q(X,D_X)$, one must recover a Riemann surface downstairs, a finite map to it, the full space of quadratic differentials downstairs, and finally the marked points. The following generic reconstruction theorem is the technical core of the proof.

\begin{theorem}\label{thm:generic-reconstruction}
Let $N\subset\mathcal{T}_{g,n}$ be a complex totally geodesic submanifold of dimension $d\ge 3$. There exists a nonempty dense open subset $N^\circ\subset N$ such that, for every $X\in N^\circ$, there are a compact Riemann surface $Y_X$, a nonconstant holomorphic map
\begin{equation} h_X:X\to Y_X, \nonumber \end{equation}
and a finite set $B_X\subset Y_X$ satisfying
\begin{equation} h_X^{-1}(B_X) = D_X\cup\operatorname{Ram}(h_X) \nonumber \end{equation}
and
\begin{equation} Q_X N = \frac{1}{\deg h_X}\, h_X^*Q(Y_X,B_X). \nonumber \end{equation}
\end{theorem}

\subsection{Idea of the proof }

The proof begins with a natural construction. The linear system determined by $Q_X N$ gives a map
\begin{equation} X\to \mathbb{P}(Q_X N^\vee).  \nonumber \end{equation}

Taking the normalization of its image gives a compact Riemann surface $Y_X$ and a nonconstant holomorphic map \begin{equation} h_X:X\to Y_X.   \nonumber \end{equation}

The essential property of this construction is that the meromorphic functions on $Y_X$ are generated, after pullback, by the ratios of elements of $Q_X N$.  Generically, the fibers of $h_X$ are precisely the equivalence classes of points that cannot be distinguished by ratios of elements of $Q_X N$.

The main part of the proof is to show that this algebraic geometric construction also carries the Teichm\"uller geometry of $N$.  The isometric identification $Q_X N\simeq T_X^*N$ gives a norm-one projection \begin{equation} P_X:Q(X,D_X)\to Q_X N.  \nonumber \end{equation} After dividing by a nonzero element of $Q_X N$, an $L^1$-projection lemma shows that $P_X$ is given, fiberwise, by a weighted average over the fibers of $h_X$.

The crucial step is to show that the weighted average is the ordinary trace average. Using the degree estimate together with the fiberwise averaging formula, we find a nonzero $q_0\in Q_X N$ whose pullback to the Galois closure is an eigenvector for the deck group. Thus
\begin{equation} g^*p^*q_0=\chi(g)p^*q_0  \nonumber \end{equation}
for a character $\chi$. We then vary along the Teichm\"uller disc generated by $q_0$; the deck group fixes the lifted disc, and the Teichm\"uller pairing forces $\chi$ to be trivial. Hence $q_0$ descends to $Y_X$, and the weighted fiberwise average becomes the ordinary normalized trace.

It remains to recover the marked points downstairs. The trace description determines the possible poles of the descended quadratic differentials. Applying the same norm-one projection argument on $Y_X$, followed by Riemann-Roch, shows that these poles occur exactly at the branch values of $h_X$ and the images of the marked points of $X$. Consequently $h_X^{-1}(B_X)=D_X\cup\operatorname{Ram}(h_X),$ and
\begin{equation}     Q_X N=\frac{1}{\deg h_X}\,h_X^*Q(Y_X,B_X).\nonumber \end{equation} This is precisely the cotangent description of a covering construction and completes the reconstruction.

\textbf{ Where the dimension hypothesis enters.  }

A useful feature of the argument is that the hypothesis $d\ge 3$ is needed only at one stage. The cotangent space $Q_X N$ and the projection $P_X$ exist for every positive-dimensional totally geodesic submanifold. When $d\ge 2$, ratios of elements of $Q_X N$ already determine a nonconstant finite map
\begin{equation} h_X:X\to Y_X, \nonumber \end{equation}
and the $L^1$ argument already describes $P_X$ as a weighted average along the fibers of $h_X$.

In particular, for a totally geodesic surface, the same construction produces
\begin{equation} h_X:X\to \mathbb{P}^1, \nonumber \end{equation}
and the projection still has a fiberwise averaging feature. Thus the reconstruction does not simply break down in complex dimension two.

The hypothesis $d\ge3$ is used only to produce a differential descending to $Y_X$.  The projective curve determined by $Q_X N$ is a line when $d=2$, while for $d\ge3$ its degree is at least two.  The degree bound is then strong enough to ensure that, over a generic point of $Y_X$, the values of a global quadratic differential can be chosen arbitrarily and independently at the different points of the fiber, except for one conic case in dimension three. This freedom allows us to compare the different sheets of the cover and find a differential whose pullback to the Galois closure is multiplied by a scalar under each deck transformation. Teichm\"uller deformation then shows that all these scalars are equal to one, so the differential descends to $Y_X$. The remaining steps do not essentially use the hypothesis $d\ge3$.

Thus, in complex dimension two, the method still recovers the quotient map and the weighted fiberwise projection, but it does not force the weights to become the ordinary trace. This gives a concrete obstruction in precisely the dimension where primitive totally geodesic examples are known to exist.

\subsection{ Related Work }  Rigidity of Teichm\"uller spaces as complex and metric spaces has a long history. Classical results of Royden and Earle and Kra show that, apart from the familiar low-complexity exceptions, biholomorphisms and isometries between finite type Teichm\"uller spaces are geometric \cite{MR288254,MR348098}. Markovi\'c developed methods for studying linear isometries between $L^1$-spaces of integrable holomorphic quadratic differentials and used them to study biholomorphic maps between Teichm\"uller spaces \cite{MR2019982}; Earle and Markovi\'c gave a corresponding rigidity result for the $L^1$-spaces of holomorphic quadratic differentials on Riemann surfaces of finite type \cite{MR2019983}. Gekhtman and Greenfield later showed, up to low genus exceptions, that holomorphic isometric submersions between Teichm\"uller spaces are forgetful maps \cite{MR4425346}. In the complementary direction, Benirschke and Serv\'an proved that holomorphic isometric embeddings between Teichm\"uller spaces are induced by covering constructions \cite{MR4768147}.

Related rigidity for isometric discs was established by Antonakoudis, who proved that every totally geodesic isometry from the unit disc into a finite-dimensional Teichm\"uller space is holomorphic or anti-holomorphic \cite{MR3608291}. There are several other indications of strong rigidity in higher-dimensional totally geodesic geometry. Benirschke associated natural linear system to generic curves in totally geodesic subvarieties and obtained rank and gonality restrictions \cite{MR4798636}. Recent work on the Deligne and Mumford boundary also reveals strong constraints on the degeneration of totally geodesic subvarieties \cite{AranaHerreraWright2025,BenirschkeDozierRached2024}.

\subsection{Organization of the paper}  Section~2 reviews quadratic differentials and cotangent spaces of totally geodesic submanifolds, covering constructions, and the $L^1$-projection lemma used later. Section~3 proves the generic reconstruction theorem in three steps. In Section~3.1, we construct the quotient surface $Y_X$ and the finite map $h_X:X\to Y_X$. In Section~3.2, we reconstruct the quadratic differentials downstairs and identify the projection with the normalized trace. In Section~3.3, we recover the marked points and verify the compatibility condition for a covering construction. Section~4 deduces Theorem~\ref{thm:main} from the generic reconstruction theorem. The appendix establishes the holomorphic variation of the quotient maps and their Galois closures.

\subsection{Acknowledgements}

The author is deeply grateful to his advisor, Gang Tian, for his constant support. He thanks Xilun Li and Yanan Ye for many helpful discussions on complex geometry, and Minghao Miao, Linsheng Wang and Ruiming Liang for their help with questions in algebraic geometry. The author is especially grateful to Alex Wright for many insightful comments on the results, proofs, and broader context of this work, as well as for numerous suggestions that substantially improved the presentation and organization of the paper.

\section{Background and preliminaries}\label{sec:background}

\subsection{Quadratic differentials and cotangent spaces of totally geodesic submanifolds}

We work throughout with compact Riemann surfaces with marked points rather than punctured surfaces. We use the same notation for a finite subset of a Riemann surface and the associated reduced effective divisor; the intended meaning will be clear from context.

For a finite set $E$ on $X$, let $Q(X,E)=H^0(X,K_X^{\otimes2}(E))$ be the space of meromorphic quadratic differentials with at most simple poles in $E$. It can be identified with the cotangent space $T_X^*\mathcal{T}(S,E)$ with pairing
\begin{equation} \langle[\mu],q\rangle=\int_X\mu q. \nonumber \end{equation}
where $[\mu]$ denotes the class of a bounded Beltrami differential modulo infinitesimally trivial ones.

The norm of $Q(X,E)$ is $\|q\|_1=\int_X|q|$, and the Teichm\"uller norm is dual to $\|\cdot\|_1$. If $q=\varphi(z)\,dz^2\ne0$, the initial unit tangent to its Teichm\"uller disc is represented off the zero and pole set by
\begin{equation} \mu_q=\frac{\overline{\varphi(z)}}{|\varphi(z)|}\frac{d\overline{z}}{dz}. \nonumber \end{equation}
This gives
\begin{equation} \label{eq:q-pairing} \langle[\mu_q],q\rangle=\|q\|_1. \end{equation}

Let $g:(S,E)\to(S,E)$ be an orientation-preserving self-homeomorphism; points of $E$ may be permuted. Then $g$ induces an action on $\mathcal{T}(S,E)$:
\begin{equation} g\cdot[Z,f]=[Z,f\circ g^{-1}]. \nonumber \end{equation}
The map $g$ also induces actions on the tangent and cotangent bundles of $\mathcal{T}(S,E)$. On tangent spaces, it's given by $g\cdot[\mu]:=(dg)_X([\mu])$. If $g$ is biholomorphic at $X$, then $g\cdot[\mu]=[(g^{-1})^*\mu]$, while the cotangent action is $q\mapsto g^*q$. Change of variables gives
\begin{equation} \label{eq:pairing-equivariant} \langle g\cdot[\mu],q\rangle=\langle[\mu],g^*q\rangle. \end{equation}

For a complex totally geodesic $N\subset\mathcal{T}_{g,n}$, let $Q_X N\subset Q(X,D_X)$ be the space of quadratic differentials whose Teichm\"uller discs lie in $N$. Benirschke and Serv\'an prove that the spaces $Q_X N$ are the fibers of a holomorphic vector subbundle over $N$ of rank $\dim_{\mathbb{C}}N$ \cite{MR4768147}. The subspace $Q_X N$ can be identified with $T^*_X N$, and we can use restriction to give a contractive projection map from $Q(X,D_X)$ to $Q_X N$. We describe it as follows:

Let $i:N\hookrightarrow\mathcal{T}_{g,n}$ be the inclusion. For $X\in N$, define the restriction map \begin{equation} \rho_X=(di_X)^*:Q(X,D_X)=T_X^*\mathcal{T}_{g,n}\to T_X^*N.   \nonumber \end{equation} Equivalently,\begin{equation} \rho_X(q)(v)=\langle v,q\rangle, \qquad q\in Q(X,D_X),\ v\in T_XN.  \nonumber \end{equation} We equip $T_X^*N$ with the quotient norm \begin{equation} \|\lambda\|_{\mathrm{quot}} =\inf\{\|q\|_1:\rho_X(q)=\lambda\}.  \nonumber \end{equation}

\begin{proposition}\label{prop:canonical-projection}
Let $\rho_X: Q(X,D_X)\to T_X^*N$ be the restriction map on cotangent spaces. The restriction
\begin{equation} \rho_X|_{Q_X N}: Q_X N\to T_X^*N \nonumber \end{equation}
is an isometric isomorphism. Consequently
\begin{equation} P_X=(\rho_X|_{Q_X N})^{-1}\rho_X: Q(X,D_X)\to Q_X N \nonumber \end{equation}
is a complex-linear projection with $\|P_X\|\le1$. If $\dim N>0$, then $\|P_X\|=1$.
\end{proposition}

\begin{proof}
Since  $T_X^*N\simeq Q(X,D_X)/\ker\rho_X$, we have
\begin{equation} \|\rho_X(q)\|_{\mathrm{quot}}=\inf_{k\in\ker\rho_X}\|q+k\|_1.  \nonumber \end{equation}
This implies $\|\rho_X(q)\|_{\mathrm{quot}}\le\|q\|_1$, so $\rho_X$ is contractive.

If $0\ne q\in Q_X N$, the initial unit tangent $[\mu_q]$ to the Teichm\"uller disc of $q$ lies in $T_XN$, and (\ref{eq:q-pairing}) gives that
\begin{equation} \|\rho_X(q)\|_{\mathrm{quot}}\ge |\langle[\mu_q],q\rangle|=\|q\|_1. \nonumber \end{equation}
Thus $\rho_X|_{Q_X N}$ is an isometry, hence injective. It is also surjective because both spaces have dimension $\dim_{\mathbb{C}}N$.

Since $(\rho_X|_{Q_X N})^{-1}$ is an isometry and $\rho_X$ is contractive,
\begin{equation} \|P_Xq\|_1\le\|q\|_1\qquad(q\in Q(X,D_X)). \nonumber \end{equation}
If $Q_X N\ne0$, then $P_Xv=v$ for $0\ne v\in Q_X N$, so $\|P_X\|=1$.
\end{proof}

\subsection{Covering constructions}
We now describe the covering constructions appearing in the main theorem and record the corresponding pullback formulas on tangent and cotangent spaces. We use the marked formulation of \cite{MR3664815}.

\begin{definition}\label{def:covering}
Let $h: S_X\to S_Y$ be a finite orientation-preserving topological branched cover, let $B\subset S_Y$ be the set of branch values, and let $D\subset S_X$ be the set of marked points. The marked point sets are compatible if
\begin{equation} \label{eq:compatibility} h^{-1}(B)=D\cup\operatorname{Ram}(h) \end{equation}
\end{definition}

For a marked Riemann surface $[Y,\varphi_Y]\in\mathcal{T}(S_Y,B)$, let $[X,\varphi_X]$ be the pullback marked Riemann surface by $h$, characterized by the requirement that $h_Y:=\varphi_Y\circ h\circ\varphi_X^{-1}:X\to Y$ is holomorphic. The map \begin{equation} \begin{aligned} f_h:\mathcal{T}(S_Y,B)     &\to \mathcal{T}(S_X,D),\\ [Y,\varphi_Y]     &\longmapsto [X,\varphi_X] \end{aligned}   \nonumber \end{equation}  is called the pullback map induced by $h$.

\begin{definition}
A totally geodesic submanifold $N\subset\mathcal{T}(S_X,D)$ is said to arise via a covering construction if there exist a marked surface $(S_Y,B)$ and a branched cover $h:(S_X,D)\to(S_Y,B)$ with compatible marked point sets such that \begin{equation} N=f_h\bigl(\mathcal{T}(S_Y,B)\bigr).   \nonumber \end{equation}
\end{definition}

\begin{lemma}\label{lem:localCompute}
Let $m=\deg h$. Under~(\ref{eq:compatibility}), $f_h$ is a holomorphic isometric embedding. If $X=f_h(Y)$ and $h_Y:X\to Y$ is the induced holomorphic branched cover, then $ (df_h)_Y([\mu])=[h_Y^*\mu],$ and the normalized pullback
\begin{equation} \frac{1}{m} h_Y^*:Q(Y,B)\to Q(X,D)  \nonumber \end{equation}
is an isometric embedding satisfying
\begin{equation} \label{eq:cotangent-right-inverse}    (df_h)_Y^* \circ\frac{1}{m} h_Y^*=\operatorname{id}_{Q(Y,B)}. \end{equation}
\end{lemma}

\begin{proof}
By \cite[Proposition~A.2]{MR4768147}, the compatibility condition
\begin{equation} h^{-1}(B)=D\cup\operatorname{Ram}(h)  \nonumber \end{equation}
implies that $f_h$ is a holomorphic isometric embedding.

We recall the local order formulas.  Suppose that $w=z^k$ is the local form of $h_Y$ at $x\in X$, and that $q=f(w)\,dw^2$ has order $a$ at $y=h_Y(x)$.  Then $h_Y^*q=f(z^k)k^2z^{2k-2}\,dz^2$, and hence
\begin{equation} \label{eq:local-order} \operatorname{ord}_x(h_Y^*q)=k(a+2)-2. \end{equation}
Thus the compatibility condition implies
\begin{equation} h_Y^*Q(Y,B)\subset Q(X,D).  \nonumber \end{equation}

If $[\mu]\in T_Y\mathcal{T}(S_Y,B)$ is represented by a Beltrami differential, then away from ramification points
\begin{equation} \label{eq:beltrami-pullback} h_Y^*\mu=(\mu\circ h_Y) \frac{\overline{h_Y'}}{h_Y'}, \qquad d(f_h)_Y([\mu])=[h_Y^*\mu]. \end{equation}
Since the ramification points form a finite set, they do not affect the $L^\infty$ class.  For $q\in Q(Y,B)$, change of variables therefore gives
\begin{equation} \label{eq:cover-pairing} \bigl\langle d(f_h)_Y([\mu]),h_Y^*q\bigr\rangle = m\langle[\mu],q\rangle, \end{equation}
where $m=\deg h$.

Similarly, $|h_Y^*q|=h_Y^*|q|$ on the ramification points, so $\|h_Y^*q\|_1=m\|q\|_1$. Hence
\begin{equation} \frac{1}{m} h_Y^*:Q(Y,B)\to Q(X,D)   \nonumber \end{equation}
is an isometric embedding. Using~(\ref{eq:cover-pairing}) and the definition of the dual map gives
\begin{equation} (df_h)_Y^* \left(\frac{1}{m} h_Y^*q\right)=q, \qquad q\in Q(Y,B).\nonumber \end{equation}
\end{proof}

\subsection{An $L^1$ projection lemma}

We next establish an $L^1$ projection lemma that turns the contractive property into an integral identity over preimages under the map determined by the range of the projection. Later, these preimages will become the fibers of the map reconstructed from $Q_X N$. In this subsection we identify $\mathbb{C}$ with $\mathbb{R}^2$.

We begin with a standard Fourier analysis proposition:

\begin{lemma}\label{Fourier}
Let
\begin{equation} K(z)= \begin{cases} \frac{\overline{z}}{|z|}, & z\neq 0,\\ 0, & z=0. \end{cases}\nonumber \end{equation}

If $\sigma$ is a finite complex Borel measure on $\mathbb{C}$ and
\begin{equation} (K*\sigma)(w)=\int_{\mathbb{C}} K(w-z)\,d\sigma(z)=0   \nonumber \end{equation}
for Lebesgue-almost every $w\in\mathbb{C}$, then $\sigma=0$.
\end{lemma}

\begin{proof}
The Fourier transform of the Poisson kernel on $\mathbb{R}^2$~\cite{MR1970295}  gives that , in the sense of tempered distributions,
\begin{equation} \widehat{|x|^{-1}}(\xi)=|\xi|^{-1}.   \nonumber \end{equation}
Since
\begin{equation} (\partial_{x_1}-i\partial_{x_2})K=|x|^{-1},   \nonumber \end{equation}
we obtain
\begin{equation} \widehat{K}(\xi) = -\frac{i}{2\pi} \frac{\xi_1+i\xi_2}{|\xi|^3},\qquad \xi\ne0.   \nonumber \end{equation}
In particular, $\widehat{K}$ does not vanish away from the origin.

Now let $a\in C_c^\infty(\mathbb{R}^2\setminus\{0\})$ and set $\psi=\mathcal{F}^{-1}a$. Since $a\widehat{K}\in C_c^\infty(\mathbb{R}^2\setminus\{0\})$, one has $\psi*K\in\mathcal{S}(\mathbb{R}^2)$. By Fubini's theorem,
\begin{equation} 0 = \psi*(K*\sigma) = (\psi*K)*\sigma . \nonumber \end{equation}
Taking Fourier transforms yields
\begin{equation} a(\xi)\widehat{K}(\xi)\widehat{\sigma}(\xi)=0.   \nonumber \end{equation}

Since $a$ is arbitrary and $\widehat{K}$ is nonzero on $\mathbb{R}^2\setminus\{0\}$, it follows that
\begin{equation} \widehat{\sigma}(\xi)=0 \qquad (\xi\neq 0).   \nonumber \end{equation}

The Fourier transform of a finite measure is continuous, so $\widehat{\sigma}(0)=0$ as well. Fourier uniqueness for finite complex measures therefore gives $\sigma=0$.
\end{proof}

\begin{lemma}\label{lem:L1-projection}
Let $(\Omega,\nu)$ be a complete finite measure space, let $A\subset L^1(\nu)$ be a finite-dimensional complex subspace containing $1$, and let
\begin{equation} \Pi:A\to A  \nonumber \end{equation}
be a complex-linear contractive projection satisfying $\Pi 1=1$. Let $V=\operatorname{Ran}\Pi$, with basis $1,g_1,\ldots,g_r$, and define \begin{equation} G=(g_1,\ldots,g_r):\Omega\to\mathbb{C}^r. \nonumber \end{equation} Then, for every $f\in A$ and every Borel set $E\subset\mathbb{C}^r$, \begin{equation} \int_{G^{-1}(E)}(f-\Pi f)\,d\nu=0.  \nonumber \end{equation} In particular, if $G$ is injective almost everywhere with respect to $\nu$, then $\Pi=\mathrm{id}$.
\end{lemma}

\begin{proof}
Fix $f\in A$ and put \begin{equation} \eta=f-\Pi f\in\ker\Pi.   \nonumber \end{equation}

We first prove the corresponding statement for the preimage of a single function $g\in V$.

The pushforward measure $g_*\nu$ has at most countably many atoms. Hence, outside a countable set of $w\in\mathbb{C}$, one has $w-g\neq 0$ $\nu$-almost everywhere. For such $w$ and every $t\in\mathbb{R}$, the contractive property of $\Pi$ gives
\begin{equation} \|w-g\|_1=  \|\Pi(w-g+t\eta)\|_1 \le  \|w-g+t\eta\|_1.\nonumber \end{equation}
Thus $t=0$ is a minimum. So we have
\begin{equation} \left.\frac{d}{dt}\right|_{t=0}|w-g+t\eta|=\operatorname{Re}\left(\frac{\overline{w-g}}{|w-g|}\,\eta\right)   \nonumber \end{equation}
almost everywhere. Since the difference quotients are bounded by $|\eta|$, dominated convergence gives
\begin{equation} \operatorname{Re} \int_\Omega \frac{\overline{w-g}}{|w-g|}\,\eta\,d\nu = 0.  \nonumber \end{equation}
Applying the same argument to $i\eta\in\ker\Pi$ yields
\begin{equation} \int_\Omega \frac{\overline{w-g}}{|w-g|}\,\eta\,d\nu = 0  \nonumber \end{equation}
for almost every $w\in\mathbb{C}$.

Define a finite complex Borel measure $\sigma_g$ on $\mathbb{C}$ by
\begin{equation} \sigma_g(E) = \int_{g^{-1}(E)}\eta\,d\nu .  \nonumber \end{equation}
Then
\begin{equation} (K*\sigma_g)(w)=\int_\Omega K(w-g)\eta\,d\nu = 0  \nonumber \end{equation}
for almost every $w\in\mathbb{C}$. By Lemma~\ref{Fourier}, $\sigma_g=0$. Hence
\begin{equation} \label{eq:one-dimensional-level-set} \int_{g^{-1}(E)}\eta\,d\nu=0 \end{equation}
for every Borel set $E\subset\mathbb{C}$ and every $g\in V$.

We now pass to  $g_1,\ldots,g_r$. Define a finite complex Borel measure $\lambda$ on $\mathbb{C}^r$ by
\begin{equation} \lambda(E) = \int_{G^{-1}(E)}\eta\,d\nu .  \nonumber \end{equation}
For $a=(a_1,\ldots,a_r)\in\mathbb{C}^r$, let \begin{equation} L_a(z_1,\ldots,z_r) = \sum_{j=1}^r a_jz_j.  \nonumber \end{equation}  Since $L_a\circ G = \sum_{j=1}^r a_jg_j \in V,$ equation (\ref{eq:one-dimensional-level-set}) implies
\begin{equation} (L_a)_*\lambda=0 \qquad\text{for every }a\in\mathbb{C}^r.  \nonumber \end{equation}

Identify $\mathbb{C}^r$ with $\mathbb{R}^{2r}$. Every real linear functional $\ell:\mathbb{R}^{2r}\to\mathbb{R}$ can be written as $\ell(z)=\operatorname{Re}L_a(z)$ for some $a\in\mathbb{C}^r$. Since $(L_a)_*\lambda=0$, we obtain
\begin{equation} \int_{\mathbb{C}^r}e^{-2\pi i\ell(z)}\,d\lambda(z)=0   \nonumber \end{equation}
for every real linear functional $\ell$. Thus the Fourier transform of $\lambda$ vanishes identically on $\mathbb{R}^{2r}$. Fourier uniqueness for finite complex measures gives $\lambda=0$, and therefore
\begin{equation} \int_{G^{-1}(E)}(f-\Pi f)\,d\nu=0  \nonumber \end{equation}
for every Borel set $E\subset\mathbb{C}^r$.
\end{proof}
\begin{remark}
Lemma~\ref{lem:L1-projection} can be regarded as a finite-dimensional subspace analogue of Douglas's theorem: a contractive projection $P$ on the full space $L^1(\nu)$ with $P1=1$ is a conditional expectation. Here $\Pi$ is defined only on a subspace $A\subset L^1(\nu)$, so Douglas's theorem does not apply directly.
\end{remark}

\section{Construction of the covering}\label{sec:construction}

Throughout this section $N\subset\mathcal{T}_{g,n}$ is totally geodesic and
\begin{equation} d=\dim_{\mathbb{C}}N\ge3. \nonumber \end{equation}

\subsection{Reconstructing the quotient surface}

Assume $\dim Q_X N\ge2$. Set $L_X=K_X^{\otimes2}(D_X)$, and let $F$ be the fixed part of the sections in $Q_X N$. Dividing by the canonical section $s_F$ gives sections of $L_X(-F)$ with no common zeros. Hence we can construct a holomorphic map
\begin{equation} \phi_X:X\to\mathbb{P}(Q_X N^\vee). \nonumber \end{equation}
Let $C_X=\phi_X(X)$, and let
\begin{equation} j_X:Y_X\to C_X \nonumber \end{equation}
be the normalization of $C_X$, which is a nondegenerate curve in $\mathbb{P}^{d-1}$.

\begin{proposition}\label{prop:ratio-quotient}
The map $\phi_X$ factors uniquely as
\begin{equation} X\xrightarrow{\ h_X\ }Y_X\xrightarrow{\ j_X\ }C_X\subset\mathbb{P}(Q_X N^\vee), \nonumber \end{equation}
where $h_X$ is a nonconstant holomorphic map. For every $0\ne q_0\in Q_X N$,
\begin{equation} \label{eq:ratio-characterization} h_X^*\mathbb{C}(Y_X)=\mathbb{C}\left(\frac{q}{q_0}:q\in Q_X N\right). \end{equation}
Moreover, setting $L_Y=j_X^*\mathcal{O}_{C_X}(1)$, we have
\begin{equation} \label{eq:descent-line-bundle} L_X(-F)\simeq h_X^*L_Y. \end{equation}
\end{proposition}

\begin{proof}
Since $\dim Q_X N\ge2$, the map $\phi_X$ is nonconstant. Its image $C_X$ is therefore an irreducible compact curve. The normalization $Y_X$ is a compact Riemann surface, and the universal property of normalization gives the unique factorization
\begin{equation} \phi_X=j_X\circ h_X. \nonumber \end{equation}
The map $h_X$ is nonconstant, hence surjective, and the normalization map $j_X$ is one-to-one away from a finite set.

Choose a basis $q_0,\ldots,q_{d-1}$ of $Q_X N$. On the affine chart where $q_0\ne0$, the coordinate functions of $\phi_X$ are
\begin{equation} \frac{q_1}{q_0},\ldots,\frac{q_{d-1}}{q_0}. \nonumber \end{equation}
Since $\phi_X=j_X\circ h_X$, there are meromorphic functions $f_i$ on $Y_X$ such that
\begin{equation} \frac{q_i}{q_0}=h_X^*f_i. \nonumber \end{equation}
The affine coordinate functions of $j_X$ generate $\mathbb{C}(Y_X)$ because $j_X$ is one-to-one outside a finite subset. Hence (\ref{eq:ratio-characterization}) holds.

Finally, the space of sections defining $\phi_X$ has no common zeros, so
\begin{equation} L_X(-F)\simeq\phi_X^*\mathcal{O}_{C_X}(1)=h_X^*j_X^*\mathcal{O}_{C_X}(1)=h_X^*L_Y. \nonumber \end{equation}
\end{proof}

We call $h_X:X\to Y_X$ the quotient map associated with $X$. So far, $Y_X$ and $h_X$ have been constructed pointwise. For the arguments in the next subsection, we need these maps to vary holomorphically with $X$, while retaining control of their topological covering data.

\begin{lemma}\label{lem:relative-quotient}
There exists a nonempty dense open subset $N^\circ\subset N$ such that the finite maps $\{h_X:X\to Y_X\mid X\in N^\circ\}$ form a holomorphic family of finite maps which is locally topologically trivial as a family of branched coverings. In particular, the degree and the ramification data are locally constant.

After further shrinking the neighborhood, the connected Galois closures of the maps $h_X$ form a holomorphic family with a fixed deck group $G$ and a fixed subgroup $G_X\le G$ corresponding to the intermediate cover $X$.
\end{lemma}

After removing a proper analytic subset, simultaneous normalization puts the $Y_X$ in a holomorphic family, and local triviality of branched coverings gives the lemma. The argument is straightforward but involves several technical details, so we defer the details to Appendix~\ref{app:relative-quotient-proof}.

\begin{remark}
For readers familiar with Hurwitz spaces, the deformation statement above is standard once the maps $h_X:X\to Y_X$ are known to vary holomorphically: locally, their topological covering type is constant, and the Galois closures vary in the same way. The only point requiring some care here is that the surfaces $Y_X$ are defined individually as normalizations. We use simultaneous normalization to put them in a holomorphic family.
\end{remark}

\subsection{Reconstructing the quadratic differentials}

We next identify the projection $P_X$ in terms of the fibers of $h_X$ and recover the quadratic differentials downstairs.

Fix $X\in N^\circ$, choose $0\ne q_0\in Q_X N$, and let
\begin{equation} h=h_X: X\to Y=Y_X \nonumber \end{equation}
be the quotient map.  The map $q\mapsto q/q_0$ identifies $Q(X,D_X)$ isometrically with the finite-dimensional subspace
\begin{equation} \left\{\frac{q}{q_0}:q\in Q(X,D_X)\right\}  \subset L^1(X,|q_0|).  \nonumber \end{equation}
By proposition~\ref{prop:canonical-projection}, under this identification, $P_X$ becomes the projection
\begin{equation} \Pi\left(\frac{q}{q_0}\right)=\frac{P_Xq}{q_0}. \nonumber \end{equation}
Its range is $\{q/q_0:q\in Q_X N\}$, it is contractive, and $\Pi1=1$.

\begin{proposition}\label{prop:fiber-average}
For every $q\in Q(X,D_X)$ and every Borel set $E\subset Y$,
\begin{equation} \label{eq:fiber-integral} \int_{h^{-1}(E)}\left(\frac{q}{q_0}-\frac{P_Xq}{q_0}\right)|q_0|=0. \end{equation}

\end{proposition}

\begin{proof}
Choose a basis $q_0,q_1,\ldots,q_{d-1}$ of $Q_X N$ and write
\begin{equation} \frac{q_i}{q_0}=h^*f_i, \qquad f_i\in\mathbb{C}(Y). \nonumber \end{equation}
Since the functions $f_1,\ldots,f_{d-1}$ generate $\mathbb{C}(Y)$, there is a finite set $S\subset Y$ containing the poles of the $f_i$ and the finitely many points where the map fails to be an immersion or injective, such that
\begin{equation} \iota=(f_1,\ldots,f_{d-1}): Y\setminus S\to\mathbb{C}^{d-1} \nonumber \end{equation}
is holomorphic, injective, and immersive. Since $q_0$ has at most simple poles, finite subsets of $X$ have $|q_0|$-measure zero; in particular, $h^{-1}(S)$ has zero $|q_0|$-measure.

In Lemma~\ref{lem:L1-projection}, take the basis
\begin{equation} 1,h^*f_1,\ldots,h^*f_{d-1} \nonumber \end{equation}
of $\operatorname{Ran}\Pi$. For every $q\in Q(X,D_X)$, let $f=q/q_0$. Take $\iota(y)=0$ for $y\in S$. Lemma~\ref{lem:L1-projection} gives
\begin{equation} \label{eq:iota-integral} \int_{(\iota\circ h)^{-1}(E')}(f-\Pi f)\,|q_0|=0 \end{equation}
for every Borel set $E'\subset\mathbb{C}^{d-1}$. If $E\subset Y\setminus S$ is Borel, then $\iota(E)$ is Borel and
\begin{equation} (\iota\circ h)^{-1}(\iota(E))=h^{-1}(E) \nonumber \end{equation}
away from $h^{-1}(S)$. Since this set is zero $|q_0|$-measure, (\ref{eq:iota-integral}) gives
\begin{equation} \int_{h^{-1}(E)}(f-\Pi f)\,|q_0|=0. \nonumber \end{equation}
The same equality holds for every Borel $E\subset Y$ after removing $S$.
\end{proof}

Let $D\subset Y$ be a coordinate disc where $h$ is unramified and $q_0$ has neither zero nor pole. Write $h^{-1}(D)=X_1\sqcup\cdots\sqcup X_m,$ and, using the coordinate $y$ on every sheet,
\begin{equation} q_0|_{X_i}=\varphi_i(y)\,dy^2, \qquad q|_{X_i}=\psi_i(y)\,dy^2. \nonumber \end{equation}

Since $P_Xq\in Q_X N$, the ratio $(P_Xq)/q_0$ is constant on the fibers of $h$. Thus on $h^{-1}(D)$ there is a meromorphic function $u_q$ on $D$ with
\begin{equation} \frac{P_Xq}{q_0}=h^*u_q. \nonumber \end{equation}
The disc contains no zero or pole of $q_0$, so $\varphi_i$ is holomorphic and does not vanish on $D$. On $X_i$,
\begin{equation} \frac{q}{q_0}=\frac{\psi_i}{\varphi_i}, \qquad |q_0|=|\varphi_i(y)|\,|dy|^2. \nonumber \end{equation}
For a Borel set $E\subset D$, equation (\ref{eq:fiber-integral}) becomes
\begin{equation} \int_E\left[\sum_{i=1}^m|\varphi_i(y)|\frac{\psi_i(y)}{\varphi_i(y)} -u_q(y)\sum_{i=1}^m|\varphi_i(y)|\right]|dy|^2=0. \nonumber \end{equation}
Since the equality above holds for every $E$, we have
\begin{equation} \label{eq:fiber-average} u_q(y)=\frac{\sum_{i=1}^m|\varphi_i(y)|\,\psi_i(y)/\varphi_i(y)}{\sum_{j=1}^m|\varphi_j(y)|}. \end{equation}

\begin{lemma}\label{lem:sheet-character}
Let $X\in N^\circ$, and let $h:X\to Y$ be the quotient map, $r:Z\to Y$ be the Galois closure given by Lemma~\ref{lem:relative-quotient}. Denote by $p:Z\to X$ the intermediate quotient, and let $G$ be the deck group of $r$. There are $0\ne q_0\in Q_X N$ and a character
\begin{equation} \chi: G\to\mathbb{C}^\times \nonumber \end{equation}
with finite image such that
\begin{equation} \label{eq:character} g^*p^*q_0=\chi(g)p^*q_0\qquad(g\in G). \end{equation}
\end{lemma}

\begin{proof}
Let $g_X$ be the genus of $X$, let $n=\deg D_X$, $m=\deg h_X$, and denote
\begin{equation} L_X=K_X^{\otimes2}(D_X). \nonumber \end{equation}
Let $F$ be the fixed part of $Q_X N$. Since $C_X$ is a nondegenerate curve in $\mathbb{P}^{d-1}$, we have
\begin{equation} \deg L_X(-F)=\deg(h_X^*L_Y)=m\deg L_Y=m\deg C_X\ge m(d-1). \nonumber \end{equation}

So we have
\begin{equation} \label{eq:degree-estimate} m(d-1)\le m\deg L_Y=\deg L_X(-F)\le4g_X-4+n. \end{equation}
Since $d-1\ge2$, we have
\begin{equation} \label{eq:m-bound} m\le2g_X-2+\frac{n}{2}. \end{equation}
Let $y\in Y$ be a regular value. Consider the short exact sequence
\begin{equation} \label{eq:restriction-sequence} 0 \to L_X(-h^{-1}(y)) \to L_X \to L_X|_{h^{-1}(y)} \to 0. \end{equation}
It gives the long exact sequence
\begin{equation} \label{eq:LES} \begin{aligned} 0 &\to H^0\bigl(X,L_X(-h^{-1}(y))\bigr) \to H^0(X,L_X) \\ &\to H^0\bigl(h^{-1}(y),L_X|_{h^{-1}(y)}\bigr) \to H^1\bigl(X,L_X(-h^{-1}(y))\bigr) \\ &\to H^1(X,L_X) \to \cdots. \end{aligned} \end{equation}
Hence, if
\begin{equation} H^1\bigl(X,L_X(-h^{-1}(y))\bigr)=0, \nonumber \end{equation}
then the connecting homomorphism vanishes, and therefore the evaluation map
\begin{equation} H^0(X,L_X) \to H^0\bigl(h^{-1}(y),L_X|_{h^{-1}(y)}\bigr) = L_X|_{h^{-1}(y)} \nonumber \end{equation}
is surjective.

By Serre duality, the obstruction is dual to
\begin{equation} H^0\bigl(X,K_X^{-1}(-D_X+h^{-1}(y))\bigr), \nonumber \end{equation}
By~\ref{eq:m-bound}, we have that 
\begin{equation}\deg\bigl(K_X^{-1}(-D_X+h^{-1}(y))\bigr)=m-(2g_X-2+n)\le -\frac{n}{2}  \nonumber \end{equation} 

If $n>0$, $\deg\bigl(K_X^{-1}(-D_X+h^{-1}(y))\bigr)$ is negative. If $n=0$, it is at most zero, and a degree-zero line bundle has a nonzero section only when it is trivial. Thus the evaluation map is not surjective only when
\begin{equation} \label{eq:exceptional-case} n=0, \qquad m=2g_X-2, \qquad \mathcal{O}_X(h^{-1}(y))\simeq K_X. \end{equation}

\textbf{The surjective case.} Suppose the evaluation map is surjective. Choose any $0\ne q_0\in Q_X N$. Let $D\subset Y$ be a coordinate disc where $h$ is unramified and $q_0$ has neither zero nor pole. Write $h^{-1}(D)=X_1\sqcup\cdots\sqcup X_m,$ and, using the coordinate $y$ on every sheet, write
\begin{equation} q_0|_{X_i}=\varphi_i(y)\,dy^2, \qquad q|_{X_i}=\psi_i(y)\,dy^2. \nonumber \end{equation}
The $\varphi_i$ are holomorphic and nonzero. Proposition~\ref{prop:fiber-average} gives
\begin{equation} \label{eq:weighted-average} \frac{P_Xq}{q_0}=h^*u_q, \qquad u_q(y)=\sum_{i=1}^m\alpha_i(y)\psi_i(y), \end{equation}
where
\begin{equation} \label{eq:alpha-weights} \alpha_i(y)=\frac{\overline{\varphi_i(y)}/|\varphi_i(y)|}{\sum_{j=1}^m|\varphi_j(y)|}. \end{equation}

Fix $y_0\in D$. Since the evaluation map is surjective, there are sections $q_1,\ldots,q_m\in H^0(X,L_X)$ whose values at $h^{-1}(y_0)$ form a basis of $L_X|_{h^{-1}(y_0)}$. Write
\begin{equation} q_j|_{X_i}=\psi_{ji}(y)\,dy^2, \qquad \Psi(y)=(\psi_{ji}(y))_{j,i}. \nonumber \end{equation}
Replacing $D$ by a smaller disc, assume that $\Psi(y)$ is invertible throughout $D$. For each $j$, $u_{q_j}$ in (\ref{eq:weighted-average}) is holomorphic on $D$ since $P_X q_j\in Q_X N$ and  $D$ contains no zero or pole of $q_0$ and no image of a marked point. Equation~(\ref{eq:alpha-weights}) gives that
\begin{equation} (\alpha_1,\ldots,\alpha_m)^T =\Psi(y)^{-1}(u_{q_1},\ldots,u_{q_m})^T, \nonumber \end{equation}
Thus every $\alpha_i$ is holomorphic.

For each $i,j$, the quotient $\alpha_i/\alpha_j$ is holomorphic and has absolute value one. It is therefore constant by the open mapping theorem. By (\ref{eq:alpha-weights}),
\begin{equation} \frac{\alpha_i}{\alpha_j} =\frac{\overline{\varphi_i/\varphi_j}}{|\varphi_i/\varphi_j|}. \nonumber \end{equation}
Since $\alpha_i/\alpha_j$ is constant, the holomorphic function $\varphi_i/\varphi_j$ has constant argument. By the open mapping theorem, $\varphi_i/\varphi_j$ is therefore constant.

For $g\in G$, let $\widetilde{D}$ be any component of $r^{-1}(D)$. Suppose that $p(\widetilde{D})=X_i$ and $p(g(\widetilde{D}))=X_j$. Then $p^*q_0=\varphi_i(y)\,dy^2$ on $\widetilde{D}$ and $p^*q_0=\varphi_j(y)\,dy^2$ on $g(\widetilde{D})$. Since $r\circ g=r$, we have $y\circ g=y$, and hence on $\widetilde{D}$
\begin{equation} g^*p^*q_0=g^*\bigl(\varphi_j(y)\,dy^2\bigr)=\varphi_j(y\circ g)\,d(y\circ g)^2=\varphi_j(y)\,dy^2. \nonumber \end{equation}
Therefore
\begin{equation} \frac{g^*p^*q_0}{p^*q_0}=\frac{\varphi_j(y)}{\varphi_i(y)}, \nonumber \end{equation}
which is constant. So the meromorphic function
\begin{equation} \frac{g^*p^*q_0}{p^*q_0} \nonumber \end{equation}
is constant on each component of $r^{-1}(D)$. Since $Z$ is connected, analytic continuation makes it a global constant, which we denote by $\chi(g)$. For $g_1,g_2\in G$,
\begin{equation} (g_1g_2)^*p^*q_0=g_2^*(g_1^*p^*q_0)=\chi(g_1)\chi(g_2)p^*q_0. \nonumber \end{equation}
Thus $\chi(g_1g_2)=\chi(g_1)\chi(g_2)$, and $\chi: G\to\mathbb{C}^\times$ is a character with finite image.

\smallskip
\textbf{The exceptional case.} If the evaluation map is not surjective, (\ref{eq:exceptional-case}) holds. Since $m>0$, (\ref{eq:degree-estimate}) gives
\begin{equation} d-1\le\deg L_Y\le2. \nonumber \end{equation}
The inequalities force $d=3$ and $\deg L_Y=2$. Equality then holds throughout (\ref{eq:degree-estimate}), so $F=0$. The complete linear system $|L_Y|$ maps $Y$ one-to-one outside a finite subset onto a nondegenerate plane conic, so
\begin{equation} Y\simeq\mathbb{P}^1, \qquad L_Y\simeq\mathcal{O}_{\mathbb{P}^1}(2). \nonumber \end{equation}
Since $\mathcal{O}_X(h^{-1}(y))\simeq K_X$ and $h^{-1}(y)$ is a fiber,
\begin{equation} K_X\simeq h^*\mathcal{O}_{\mathbb{P}^1}(1). \nonumber \end{equation}
The pullback of $H^0(\mathbb{P}^1,\mathcal{O}(1))$ gives a base-point-free two-dimensional subspace \begin{equation} V\subset H^0(X,K_X).  \nonumber \end{equation} Choose a basis $\omega_0,\omega_1$ of $V$ so that $h=\omega_1/\omega_0$. Then
\begin{equation} Q_X N=h^*H^0(\mathbb{P}^1,\mathcal{O}(2))=\operatorname{Sym}^2\langle\omega_0,\omega_1\rangle. \nonumber \end{equation}
Choose $q_0=\omega_0^2$. Let $D\subset Y$ be a coordinate disc where $h$ is unramified and $q_0$ has neither zero nor pole. Write $h^{-1}(D)=X_1\sqcup\cdots\sqcup X_m,$ and, using the coordinate $y$ on every sheet,
\begin{equation} \omega_0|_{X_i}=a_i(y)\,dy, \qquad q|_{X_i}=\psi_i(y)\,dy^2. \nonumber \end{equation}

By the long exact sequence~(\ref{eq:LES}), the kernel of the evaluation map is
\begin{equation} H^0\bigl(X,K_X^{\otimes2}(-h^{-1}(y))\bigr)\simeq H^0(X,K_X). \nonumber \end{equation}
Therefore the rank of the evaluation map is
\begin{equation} (3g_X-3)-g_X=2g_X-3=m-1. \nonumber \end{equation}

Now consider the meromorphic one form $\theta_y=\frac{q}{(h-y)\omega_0}$. It has simple poles on the fiber $h^{-1}(y)$ and no other poles. By the residue theorem,
\begin{equation} \label{eq:residue} 0=\sum_{x\in h^{-1}(y)}\operatorname{Res}_x(\theta_y)=\sum_{i=1}^{m}\frac{\psi_i(y)}{a_i(y)}. \end{equation}
Thus
\begin{equation} \psi_m(y)=-a_m(y)\sum_{i=1}^{m-1}\frac{\psi_i(y)}{a_i(y)}   \nonumber \end{equation}
Plugging into~(\ref{eq:weighted-average}), we obtain
\begin{equation} \label{eq:except-weights} u_q(y)=\sum_{i=1}^{m-1}\bigl(\alpha_i(y)-\frac{a_m}{a_i}\alpha_m\bigr)\psi_i(y). \end{equation}

As in the surjective case, there are sections $q_1,\ldots,q_{m-1}\in H^0(X,L_X)$ whose values at $h^{-1}(y_0)$ form a basis of the image of the evaluation map. It follows that $\bigl(\alpha_i(y)-\frac{a_m}{a_i}\alpha_m \bigr)$ is holomorphic.

For $q_0=\omega_0^2$, formula (\ref{eq:alpha-weights}) becomes
\begin{equation} \label{eq:exceptional-alpha} \alpha_i=\frac{\overline{a_i}}{a_i\sum_{j=1}^m|a_j|^2}. \end{equation}
Hence
\begin{equation} \bigl(\alpha_i(y)-\frac{a_m}{a_i}\alpha_m \bigr)=\frac{\overline{a_i}-\overline{a_m}}{\sum_{\ell=1}^m|a_\ell|^2}\nonumber \end{equation}
is holomorphic. Moreover
\begin{equation} \frac{\overline{a_i}-\overline{a_j}}{\sum_{\ell=1}^m|a_\ell|^2}\nonumber \end{equation}
is holomorphic. The pushforward of $\omega_0$ along $h$, defined on the unramified set by summing over the local inverse branches and extended to all of $Y$, gives a holomorphic one-form on $Y$. Since $Y\simeq\mathbb{P}^1$, $h_*\omega_0=0$. Thus on $D$
\begin{equation} \label{eq:sum-ai-zero} \sum_{i=1}^m a_i=0. \end{equation}
The $a_i$ are nonzero, so they are not all the same. Choose $i_0,j_0$ with $a_{i_0}-a_{j_0}$ nonzero and shrink $D$ so that it never vanishes. Thus
\begin{equation} \frac{\overline{a_{i_0}}-\overline{a_{j_0}}}{\sum_{\ell=1}^m|a_\ell|^2}  \nonumber \end{equation}
is holomorphic and nonzero, which implies that the quotient
\begin{equation} \frac{\overline{a_{i}}-\overline{a_{j}}}{\overline{a_{i_0}}-\overline{a_{j_0}}}   \nonumber \end{equation}
is holomorphic. But on the other hand, the quotient
\begin{equation} \frac{a_i-a_j}{a_{i_0}-a_{j_0}} \nonumber \end{equation}
is also holomorphic.

Hence it is constant. Write
\begin{equation} a_i=(a_{i_0}-a_{j_0})c_i+a_{j_0}, \qquad c_i=\frac{a_i-a_{j_0}}{a_{i_0}-a_{j_0}}\in\mathbb{C}. \nonumber \end{equation}
The $c_i$ are constant. Summing over $i$ we get that
\begin{equation} 0= \sum_{i=1}^{m} a_i=(a_{i_0}-a_{j_0})\sum_{i=1}^{m}c_i+ma_{j_0},  \nonumber \end{equation}
Thus $a_{j_0}$ is a constant multiple of $a_{i_0}-a_{j_0}$. Absorbing that constant into the $c_i$ gives
\begin{equation} \label{eq:ai-proportional} a_i=(a_{i_0}-a_{j_0})c_i, \qquad c_i\in\mathbb{C}^\times. \end{equation}

For $g\in G$, the ratio $g^*p^*\omega_0/p^*\omega_0$ is constant on every component of $r^{-1}(D)$ by (\ref{eq:ai-proportional}), hence constant on connected $Z$. Denote this constant by $\chi_0(g)$. Pullback by a product $g_1g_2$ shows that $\chi_0: G\to\mathbb{C}^\times$ is a character. Its image is finite because $G$ is finite. Squaring gives (\ref{eq:character}) for $q_0=\omega_0^2$ with $\chi=\chi_0^2$.
\end{proof}

\begin{lemma}\label{lem:character-trivial}
The differential $q_0$ in Lemma~\ref{lem:sheet-character} is an ordinary pullback: there is a meromorphic quadratic differential $\eta$ on $Y$ such that
\begin{equation} q_0=h^*\eta. \nonumber \end{equation}
\end{lemma}

\begin{proof}
Restrict the family in Lemma~\ref{lem:relative-quotient} to a small disc in the Teichm\"uller disc generated by $q_0$, centered at $X$ and contained in $N^\circ$. Denote this restricted Teichm\"uller disc by
\begin{equation} \gamma:\Delta\to\mathcal{T}(S_X,D_X), \qquad \gamma(0)=X, \nonumber \end{equation}
and choose its parameter so that
\begin{equation} v_X =d\gamma_0(1)=[\mu_{q_0}]. \nonumber \end{equation}
By (\ref{eq:q-pairing}),
\begin{equation} \langle v_X,q_0\rangle=\|q_0\|_1>0. \nonumber \end{equation}

On the central fiber set
\begin{equation} E_Y=\operatorname{Br}(r)\cup h(D_X), \qquad E_X=h^{-1}(E_Y), \qquad E_Z=r^{-1}(E_Y)=p^{-1}(E_X). \nonumber \end{equation}
After shrinking $\Delta$, Lemma~\ref{lem:relative-quotient} gives holomorphic sections whose values in each fiber are exactly these points. The family of Riemann surfaces with the additional marked points therefore defines a lift
\begin{equation} \widehat{\gamma}:\Delta\to\mathcal{T}(S_X,E_X) \nonumber \end{equation}
of $\gamma$ under the forgetful map $\mathcal{T}(S_X,E_X)\to\mathcal{T}(S_X,D_X)$. Put $\widehat{v}_X=d\widehat{\gamma}_0(1)$. The cotangent map of the forgetful map is the natural inclusion
\begin{equation} Q(X,D_X)\hookrightarrow Q(X,E_X). \nonumber \end{equation}
Therefore
\begin{equation} \label{eq:marked-pairing} \langle\widehat{v}_X,q_0\rangle=\langle v_X,q_0\rangle=\|q_0\|_1. \end{equation}

Every branch value of $p$ lies in $E_X$. If $z$ is ramified for $p$ and $x=p(z)$, then the local degree of $r=h\circ p$ at $z$ is greater than one, so $h(x)\in\operatorname{Br}(r)$ and $x\in h^{-1}(E_Y)=E_X$. Thus the fixed topological branched cover
\begin{equation} p:(S_Z,E_Z)\to(S_X,E_X) \nonumber \end{equation}
defines a covering map
\begin{equation} f_p:\mathcal{T}(S_X,E_X)\to\mathcal{T}(S_Z,E_Z). \nonumber \end{equation}
Let
\begin{equation} \widetilde{\gamma}=f_p\circ\widehat{\gamma}, \qquad v_Z=d\widetilde{\gamma}_0(1). \nonumber \end{equation}
Equations (\ref{eq:beltrami-pullback}), (\ref{eq:cover-pairing}), and (\ref{eq:marked-pairing}) give
\begin{equation} \label{eq:upstairs-pairing} v_Z=df_p(\widehat{v}_X), \qquad \langle v_Z,p^*q_0\rangle=(\deg p)\|q_0\|_1>0. \end{equation}

For every $g\in G$ and every $t\in\Delta$, the deck transformation $g$ is biholomorphic on $Z_t$ and preserves $E_Z$. Hence the action fixes the lifted Teichm\"uller disc:
\begin{equation} g\cdot\widetilde{\gamma}(t)=\widetilde{\gamma}(t). \nonumber \end{equation}
Differentiating at $t=0$ gives
\begin{equation} \label{eq:invariant-tangent} g\cdot v_Z=v_Z. \end{equation}
Using (\ref{eq:pairing-equivariant}), (\ref{eq:invariant-tangent}), and (\ref{eq:character}), we obtain
\begin{equation} \langle v_Z,p^*q_0\rangle =\langle g\cdot v_Z,p^*q_0\rangle =\langle v_Z,g^*p^*q_0\rangle =\chi(g)\langle v_Z,p^*q_0\rangle. \nonumber \end{equation}
The pairing in (\ref{eq:upstairs-pairing}) is nonzero, so $\chi(g)=1$ for every $g\in G$.

Thus $p^*q_0$ is invariant under the deck group of $r: Z\to Y$. Hence there is a meromorphic quadratic differential $\eta$ on $Y$ with
\begin{equation} p^*q_0=r^*\eta=p^*h^*\eta. \nonumber \end{equation}
Since $p$ is surjective, $p^*$ is injective on meromorphic quadratic differentials, so $q_0=h^*\eta$.
\end{proof}

For a meromorphic quadratic differential $q$ on $X$, define the pushforward $h_*q$ away from the branch values by summing its pullbacks under the local inverse branches of $h$. This extends uniquely to a meromorphic quadratic differential on $Y$.

\begin{proposition}\label{prop:trace-reconstruction}
Let $X\in N^\circ$, let $h: X\to Y$ be the quotient map, and define
\begin{equation} D_Y=\{y\in Y:h^{-1}(y)\subset D_X \cup \operatorname{Ram}(h) \}. \nonumber \end{equation}
Then $D_Y$ is finite. Writing $m=\deg h$,
\begin{equation} \label{eq:trace-reconstruction} P_X=\frac{1}{m} h^*h_*, \qquad Q_X N=\frac{1}{m} h^*Q(Y,D_Y). \end{equation}
\end{proposition}

\begin{proof}
Since $\operatorname{Br}(h)\cup h(D_X)$ is finite and \begin{equation} D_Y\subset h\bigl(h^{-1}(D_Y)\bigr)\subset h(D_X \cup \operatorname{Ram}(h))=\operatorname{Br}(h)\cup h(D_X), \nonumber \end{equation} the set $D_Y$ is finite.

By Lemma~\ref{lem:character-trivial}, choose $0\ne q_0\in Q_X N$ with $q_0=h^*\eta$. Let $D\subset Y$ be a coordinate disc containing no branch value and no zero or pole of $\eta$, and write $\eta=\varphi(y)\,dy^2$. On every inverse sheet the coefficient of $q_0$ is the same $\varphi(y)$. (\ref{eq:alpha-weights}) becomes
\begin{equation} \alpha_i(y)=\frac{1}{m\varphi(y)}. \nonumber \end{equation}
If $q|_{X_i}=\psi_i(y)\,dy^2$, then (\ref{eq:weighted-average}) gives on $h^{-1}(D)$
\begin{equation} P_Xq=\frac{1}{m} h^*\left(\sum_i\psi_i(y)\,dy^2\right) =\frac{1}{m} h^*h_*q. \nonumber \end{equation}
Both sides are meromorphic quadratic differentials, so the identity holds globally.

Now we determine the image of
\begin{equation} h_*:Q(X,D_X)\to \{\text{meromorphic quadratic differentials on }Y\}.     \nonumber \end{equation}
Since
\begin{equation} P_X=\frac{1}{m} h^*h_* \qquad\text{and}\qquad h_*h^*=m\,\mathrm{id},  \nonumber \end{equation}
we have
\begin{equation} h_*Q(X,D_X) =\{u:h^*u\in Q(X,D_X)\}.    \nonumber \end{equation}

It remains to characterize the differentials on the right.  Let $y=h(x)$ and let $k$ be the local degree of $h$ at $x$.  If $u$ has order $a$ at $y$,
\begin{equation} \operatorname{ord}_x(h^*u)=k(a+2)-2.    \nonumber \end{equation}
Thus there is no pole of order at least two.  A simple pole at $y$ appears precisely when every unramified point of $h^{-1}(y)$ belongs to $D_X$, or equivalently when $y\in D_Y$.  Hence
\begin{equation} h_*Q(X,D_X)=Q(Y,D_Y).   \nonumber \end{equation}
Since $\operatorname{Ran}P_X=Q_X N$, we conclude that
\begin{equation} Q_X N=\frac{1}{m} h^*Q(Y,D_Y). \nonumber \end{equation}
\end{proof}

\subsection{Reconstructing the marked points}

\begin{proposition}\label{prop:marked-completion}
Let $X\in N^\circ$, let $h: X\to Y$ be the quotient map, and let \begin{equation} B=\operatorname{Br}(h)\cup h(D_X). \nonumber \end{equation} Then
\begin{equation} \label{eq:B-equals-DY} B=D_Y \end{equation}
and \begin{equation} \label{eq:marked-cover} h^{-1}(B)=D_X\cup\operatorname{Ram}(h). \end{equation} Thus $h$ defines a covering construction and its pullback cotangent subspace is $Q_X N$.
\end{proposition}

\begin{proof}
Fix $X\in N^\circ$, write $h:X\to Y$ for the quotient map, and let $m=\deg h$. Choose a simply connected neighborhood $U\subset N^\circ$ on which Lemma~\ref{lem:relative-quotient} gives a fixed topological branched cover $h$ and constant marking sections. The family of quotient surfaces over $U$, marked by the branch values and the images of the marked points, defines a holomorphic map
\begin{equation} c_Y:U\to\mathcal{T}(Y,B). \nonumber \end{equation}
Pullback by $h$, followed by forgetting the additional marked points in $h^{-1}(B)\setminus D_X$, defines a holomorphic map
\begin{equation} f_h:\mathcal{T}(Y,B)\to\mathcal{T}(X,D_X). \nonumber \end{equation}
For each $t\in U$, the map $H_t:X_t\to Y_t$ have the same underlying topological branched cover as $h$. For each $t\in U$, pulling back the marked Riemann surface represented by $c_Y(t)$ along $h$ and then forgetting the points in $h^{-1}(B)\setminus D_X$ gives the original marked surface $X_t$. Hence
\begin{equation}
f_h(c_Y(t))=\iota(t).
\nonumber
\end{equation}
Here $\iota:U\hookrightarrow\mathcal{T}(X,D_X)$ is the inclusion. Hence $c_Y$ is an immersion; after shrinking $U$, we regard $c_Y(U)$ as a $d$-dimensional complex submanifold of $\mathcal{T}(Y,B)$.

At the central fiber set
\begin{equation} V_Y=Q(Y,D_Y)\subset Q(Y,B). \nonumber \end{equation}
By Proposition~\ref{prop:trace-reconstruction},
\begin{equation} Q_X N=\frac{1}{m} h^*V_Y, \nonumber \end{equation}
and the normalized pullback is an $L^1$-isometry. In particular, $\dim V_Y=d$.

Let
\begin{equation} \rho_Y:Q(Y,B)\to T_Y^*c_Y(U) \nonumber \end{equation}
be cotangent restriction. We claim that $\rho_Y|_{V_Y}$ is an isometry. Indeed, if $0\ne\eta\in V_Y$ and
\begin{equation} q=\frac{1}{m} h^*\eta\in Q_X N, \nonumber \end{equation}
then a small neighborhood of the Teichm\"uller disc determined by $q$ lies in $U$. Since $q=\frac{1}{m}h^*\eta$, pullback by $h$ identifies the Teichm\"uller deformation determined by $\eta$ with that determined by $q$; hence $c_Y$ maps the latter locally to the former. Thus the unit Teichm\"uller tangent determined by $\eta$ lies in $T_Yc_Y(U)$. The same duality argument as in Proposition~\ref{prop:canonical-projection} gives
\begin{equation} \|\rho_Y(\eta)\|=\|\eta\|_1. \nonumber \end{equation}
Since $\dim V_Y=\dim c_Y(U)=d$, the restriction $\rho_Y|_{V_Y}$ is an isometric isomorphism. Therefore
\begin{equation} P_Y=(\rho_Y|_{V_Y})^{-1}\circ\rho_Y:Q(Y,B)\to V_Y \nonumber \end{equation}
is a norm-one projection.

We now show that $P_Y$ is the identity. Choose a basis $\eta_0,\ldots,\eta_{d-1}$ of $V_Y$ with $\eta_0\ne0$, and set $f_i=\eta_i/\eta_0$. The ratios of elements of $V_Y$ generate $\mathbb{C}(Y)$: after pullback by $h$, they generate the ratio field of $Q_X N$, which is $h^*\mathbb{C}(Y)$ by (\ref{eq:ratio-characterization}). Hence, after removing finitely many points, the map
\begin{equation} F=(f_1,\ldots,f_{d-1}) \nonumber \end{equation}
is one-to-one onto its image.

After dividing by $\eta_0$, $Q(Y,B)$ can be identified isometrically with
\begin{equation} A=\left\{\frac{\eta}{\eta_0}:\eta\in Q(Y,B)\right\}\subset L^1(Y,|\eta_0|). \nonumber \end{equation}
Under this identification, $P_Y$ becomes the contractive projection
\begin{equation} \Pi\left(\frac{\eta}{\eta_0}\right)=\frac{P_Y\eta}{\eta_0}, \nonumber \end{equation}
whose range is spanned by $1,f_1,\ldots,f_{d-1}$. Lemma~\ref{lem:L1-projection}, together with the generic injectivity of $F$, implies
\begin{equation} f=\Pi f\qquad\text{a.e. on }Y \nonumber \end{equation}
for every $f\in A$. Hence $P_Y=\mathrm{id}$, and therefore
\begin{equation} Q(Y,D_Y)=Q(Y,B). \label{eq:Q-equality} \end{equation}

Riemann-Roch theorem gives, for $D=D_Y$ and $D=B$, 
\begin{equation} h^0 \bigl(K_Y^{\otimes 2}(D)\bigr) =
\begin{cases}
3g(Y)-3+\deg D, & g(Y)\ge 2,\\[2mm]
\deg D, & g(Y)=1,\ \deg D>0,\\[2mm]
1, & g(Y)=1,\ D=0,\\[2mm]
\max\{\deg D-3,\,0\}, & g(Y)=0.
\end{cases}    \nonumber \end{equation} 
Since $h^0\bigl(K_Y^{\otimes 2}(D)\bigr)\ge 3,$, the exceptional cases in genus 0 and 1 do not occur. Hence, in the cases here, \begin{equation} h^0\bigl(K_Y^{\otimes2}(D)\bigr)=3g(Y)-3+\deg D. \nonumber \end{equation}
Together with $Q(Y,D_Y)=Q(Y,B)$, this implies $\deg D_Y=\deg B$. Since $D_Y\subset B$ and both are reduced, (\ref{eq:B-equals-DY}) follows.

Finally, $B$ contains $h(D_X)$ and every branch value, so
\begin{equation} D_X\cup\operatorname{Ram}(h)\subset h^{-1}(B). \nonumber \end{equation}
Conversely, if $x\in h^{-1}(B)=h^{-1}(D_Y)$ is not ramified, then the definition of $D_Y$ gives $x\in D_X$. Hence
\begin{equation} h^{-1}(B)=D_X\cup\operatorname{Ram}(h). \end{equation}
The asserted description of $Q_X N$ is already given by Proposition~\ref{prop:trace-reconstruction}.
\end{proof}
\begin{proof}[Proof of Theorem~\ref{thm:generic-reconstruction}]
Take the nonempty dense open set $N^\circ$ from Lemma~\ref{lem:relative-quotient}. For $X\in N^\circ$, let $h: X\to Y$ be the quotient map and set $B=\operatorname{Br}(h)\cup h(D_X)$. Proposition~\ref{prop:marked-completion} gives
\begin{equation} h^{-1}(B)=D_X\cup\operatorname{Ram}(h) \nonumber \end{equation}
and, together with Proposition~\ref{prop:trace-reconstruction},
\begin{equation} Q_X N=\frac{1}{\deg h}h^*Q(Y,B). \nonumber \end{equation}
\end{proof}

\section{Proof of Theorem~\ref{thm:main}}\label{sec:main-proof}

\begin{proof}
Choose $X\in N^\circ$ and let $h:(X,D_X)\to(Y,B)$ be given by Theorem~\ref{thm:generic-reconstruction}. The equality
\begin{equation} h^{-1}(B)=D_X\cup\operatorname{Ram}(h) \nonumber \end{equation}
is the compatibility condition in Definition~\ref{def:covering}. Therefore the underlying topological branched cover defines a covering construction
\begin{equation} f_h:\mathcal{T}(Y,B)\to\mathcal{T}_{g,n} \nonumber \end{equation}
whose image contains $X$. By Proposition~\ref{prop:trace-reconstruction}, the cotangent subspace at $X$ is
\begin{equation} \frac{1}{\deg h}h^*Q(Y,B)=Q_X N. \nonumber \end{equation}
Let $N_h=f_h(\mathcal{T}(Y,B))$. If $Z\in N$, the Teichm\"uller disc through $X$ and $Z$ is generated at $X$ by an element of $Q_X N$, hence by an element of the cotangent subspace of $N_h$; the covering construction is isometric, so this disc lies in $N_h$. Thus $N\subset N_h$. Conversely, if $Z\in N_h$, the disc through $X$ and $Z$ is generated by an element of the same subspace $Q_X N$, and therefore lies in $N$ by the definition of $Q_X N$. Hence $N_h\subset N$ and
\begin{equation} N=f_h\bigl(\mathcal{T}(Y,B)\bigr). \nonumber \end{equation}
\end{proof}

\appendix
\section{Holomorphic families of the quotient}\label{app:relative-quotient-proof}

In the appendix, we prove Lemma~\ref{lem:relative-quotient}, which shows that the quotient maps $h_X:X\to Y_X$ vary holomorphically and locally topologically trivially over a dense open subset of $N$.

We will repeatedly remove proper analytic subsets of the base. Frisch's generic flatness theorem \cite{MR222336}, together with Remmert's proper mapping theorem \cite{Demailly2012}, allows us first to arrange that the relative projective image is flat. Standard results on families of complex spaces then allow us to remove a further proper analytic subset so that all fibers are reduced \cite{MR430286}. For a flat family of reduced curves, the locus where the fibers are not normal is analytic. Since our family is proper, this locus is proper over the base.  The simultaneous normalization theorem \cite{MR2225697} then implies that, after one restriction of the base, the normalization of the total space normalizes every fiber.

First we recall the analytic results used in the normalization argument.

\begin{theorem}\cite{Demailly2012}\label{thm:remmert}
Let $f:X\to S$ be a proper holomorphic map of complex spaces. If $A\subset X$ is a closed analytic subset, then $f(A)$ is a closed analytic subset of $S$.
\end{theorem}

\begin{theorem}\cite{MR222336}\label{thm:flatness}
Let $f:X\to S$ be a holomorphic map of complex spaces and let $\mathcal{F}$ be a coherent $\mathcal{O}_X$-module. Then
\begin{equation} \Sigma(\mathcal{F})=\{x\in X:\mathcal{F}_x\text{ is not flat over }\mathcal{O}_{S,f(x)}\} \nonumber \end{equation}
is a closed analytic subset of $X$. If $X$ is reduced, then $f(\Sigma(\mathcal{F}))$ is nowhere dense in $S$.
\end{theorem}

\begin{theorem}\cite{MR430286}\label{thm:reducedness}
Let $f:X\to S$ be a flat holomorphic map of complex spaces. The loci
\begin{equation} \{x\in X:X_{f(x)}\text{ is reduced at }x\}, \qquad \{x\in X:X_{f(x)}\text{ is normal at }x\} \nonumber \end{equation}
are open in $X$. If $f$ is surjective, $S$ is a complex manifold, and $X$ is reduced, then there is a nowhere-dense analytic subset $T\subset S$ such that $X_s$ is reduced for every $s\in S\setminus T$. If all fibers are reduced, then
\begin{equation} \operatorname{NNor}(f)=\{x\in X:X_{f(x)}\text{ is not normal at }x\} \nonumber \end{equation}
is an analytic subset of $X$.
\end{theorem}

For simultaneous normalization we use \cite[Corollary~5.4.3]{MR2225697} in the following form.

\begin{theorem}\cite{MR2225697}\label{thm:simultaneous normalization}
Let $f:X\to S$ be a flat holomorphic map, where $S$ is normal, $X$ is reduced and locally equidimensional, and the fibers of $f$ are reduced curves. Suppose that $\operatorname{NNor}(f)$ is proper over $S$. Then there is a nowhere-dense analytic subset $T\subset S$ such that the normalization
\begin{equation} \nu:\widetilde{X}\to X \nonumber \end{equation}
restricts over $S\setminus T$ to a simultaneous normalization. In particular, $f\circ\nu$ is flat with normal fibers over $S\setminus T$, and for every $s\in S\setminus T$ the induced map
\begin{equation} \nu_s:\widetilde{X}_s\to X_s \nonumber \end{equation}
is the normalization of $X_s$.
\end{theorem}

Now let $H:\mathcal{X}\to\mathcal{Y}$ be a holomorphic map over $U$ between proper holomorphic families of compact Riemann surfaces such that each fiber map $H_t:X_t\to Y_t$ is nonconstant. Since the families are proper, $H$ is a finite holomorphic map.

We first show that, after removing a proper analytic subset of the base, the ramification points, branch values, and their inverse images vary in disjoint holomorphic sections.

Let $R_H$ be the ramification divisor of $H$ relative to $U$, i.e. the zero divisor of \begin{equation} dH:H^*\Omega^1_{\mathcal{Y}/U}\to \Omega^1_{\mathcal{X}/U}.  \nonumber \end{equation} Its restriction to each fiber is the usual ramification divisor of $H_t:X_t\to Y_t$. Let $B_H=\sum_{i=1}^r m_i b_i(U)$ be an effective divisor on $\mathcal{Y}$ supported on finitely many holomorphic sections $b_i:U\to\mathcal{Y}$. Then $H^*B_H$ contains no fiber of $\mathcal{X}\to U$. Hence neither the ramification divisor $R_H$ nor the pullback $H^*B_H$ contains a fiber of $\mathcal{X}\to U$. The following lemma applies to both situations.

\begin{lemma}\label{lem:generic-splitting}
Let $\pi:\mathcal{X}\to U$ be a proper holomorphic submersion with compact connected Riemann surface fibers, where $U$ is a connected complex manifold. Let $\mathcal{R}$ be an effective divisor on $\mathcal{X}$ containing no fiber of $\pi$.

Then there exists a proper analytic subset $A\subset U$ such that
\begin{equation} \operatorname{Supp}(\mathcal{R})_{\mathrm{red}} \cap\pi^{-1}(U\setminus A) \to U\setminus A   \nonumber \end{equation}
is a finite unramified covering.

Consequently, if $V\subset U\setminus A$ is simply connected, there are pairwise disjoint holomorphic sections $s_1,\ldots,s_r:V\to\mathcal{X}$ and positive integers $m_1,\ldots,m_r$ such that \begin{equation} \label{eq:divisor-splitting}   \mathcal{R}|_{\pi^{-1}(V)}= \sum_{i=1}^r m_i\,s_i(V). \end{equation} For finitely many such divisors, $A$ may be chosen uniformly.
\end{lemma}

\begin{proof}
Set $S=\operatorname{Supp}(\mathcal{R})_{\mathrm{red}}$. Since $\mathcal{R}$ contains no fiber, $S\to U$ is proper with finite fibers and hence finite. Every irreducible component of $S$ dominates $U$ and is generically unramified. Let $S_{\mathrm{sing}}$ be the singular locus of $S$ and let $S_{\mathrm{crit}}$ be the critical locus of $S_{\mathrm{reg}}\to U$. By Theorem~\ref{thm:remmert},
\begin{equation} \pi(S_{\mathrm{sing}}\cup S_{\mathrm{crit}}) \nonumber \end{equation}
is a proper analytic subset of $U$. Away from it, $S\to U$ is finite and unramified. Over a simply connected open set, every finite unramified cover is a disjoint union of trivial covers, which gives the sections $s_i$. The coefficient of $\mathcal{R}$ along each irreducible section $s_i(V)$ is constant, which gives~(\ref{eq:divisor-splitting}). For finitely many divisors one takes the union of the corresponding exceptional subsets.
\end{proof}

Once the branch values, their inverse images, and the local degrees have been fixed in families, the branched covering itself is locally constant topologically:

\begin{lemma}\label{lem:topological-triviality}
Let $H:\mathcal{X}\to\mathcal{Y}$ be a holomorphic family of finite maps of compact Riemann surfaces over a polydisc $U$. Suppose that the branch values and all their inverse images are given by pairwise disjoint holomorphic sections, and that the local degree of $H$ along each inverse-image section is constant.

Then, after shrinking $U$ around any point $t_0$, there are smooth trivializations \begin{equation} \mathcal{X}\simeq U\times X_{t_0}, \qquad \mathcal{Y}\simeq U\times Y_{t_0},  \nonumber \end{equation} under which $H=\operatorname{id}_U\times H_{t_0}.$

The trivializations may also be chosen to make finitely many additional sections constant, provided their coincidence relations with the above sections are constant on $U$.
\end{lemma}

\begin{proof}
Fix $t_0\in U$. After replacing $U$ by a sufficiently small neighborhood of $t_0$, we may make the following constructions uniformly. Near a section $s:U\to\mathcal{X}$ lying over a branch section $b:U\to\mathcal{Y}$ with local degree $k$, choose holomorphic coordinates $(t,z)$ on $\mathcal{X}$ and $(t,w)$ on $\mathcal{Y}$, compatible with the projections to $U$, such that the two sections are given by $z=0$ and $w=0$, respectively. Then $w=z^k u$ for a nowhere-vanishing holomorphic function $u$. After shrinking the coordinate neighborhood, write $u=v^k$ with $v$ holomorphic and nowhere zero. Replacing $z$ by $vz$, we obtain the local model $w=z^k$.

For each branch section $b_j$, let $s_{j,1},\ldots,s_{j,r_j}$ be the sections satisfying $H\circ s_{j,\ell}=b_j$. After shrinking $U$, choose pairwise disjoint neighborhoods $V_j$ of $b_j(U)$ and $W_{j,\ell}$ of $s_{j,\ell}(U)$, whose intersections with the fibers are discs. Shrinking these neighborhoods further if necessary, properness of $H$ ensures that $H^{-1}(V_j)=\bigcup_{\ell=1}^{r_j}W_{j,\ell}$. Removing their interiors gives an unramified covering
\begin{equation} H^\circ:\mathcal{X}^\circ\to\mathcal{Y}^\circ.    \nonumber \end{equation}
By Ehresmann's theorem, \begin{equation} \mathcal{Y}^\circ\simeq U\times Y_{t_0}^\circ.   \nonumber \end{equation} Since $U$ is contractible, covering space theory identifies $H^\circ$ with the pullback of the central covering $H_{t_0}^\circ$. Hence $\mathcal{X}^\circ$ is trivialized so that $H^\circ$ is constant in $U$.

On each boundary circle, the trivialization obtained on the complement may differ from the one coming from the local model $w=z^k$. These two trivializations are related by boundary diffeomorphisms compatible with $z\mapsto z^k$. Using isotopy extension and the covering homotopy property, we modify the trivializations near the boundary so that they agree and hence glue.

Any additional section that coincides with one of the branch or inverse-image sections is already constant in the above trivializations. For the remaining sections, the trivialization may be chosen so that they are constant, and the subsequent boundary modifications can be supported away from them.
\end{proof}

The Galois-closure argument uses the following lemma.

\begin{lemma}\label{lem:fiberwise-holomorphic}
Let $H_i:\mathcal{X}_i\to\mathcal{Y}$, $i=1,2$, be finite holomorphic maps between holomorphic families of Riemann surfaces over a complex manifold $U$. Suppose that a fiber-preserving smooth diffeomorphism
\begin{equation} F:\mathcal{X}_1\to\mathcal{X}_2 \nonumber \end{equation}
satisfies $H_2\circ F=H_1$ and restricts to a biholomorphism on every fiber. Then $F$ is biholomorphic.
\end{lemma}

\begin{proof}
On the unramified locus of $H_1$, the map $H_2$ is also unramified at the corresponding point, because $F$ is biholomorphic on each fiber and $H_2\circ F=H_1$. Thus, using the local inverse branch of $H_2$,
\begin{equation} F=(H_2|_V)^{-1}\circ H_1, \nonumber \end{equation}
so $F$ is holomorphic on the complement of the ramification locus of $H_1$. This complement is dense. In local holomorphic coordinates the antiholomorphic derivatives of the smooth map $F$ are continuous and vanish on this dense open set; hence they vanish everywhere. Thus $F$ is holomorphic. Applying the same argument to $F^{-1}$ proves that $F$ is biholomorphic.
\end{proof}

\begin{proof}[Proof of Lemma~\ref{lem:relative-quotient}]
Choose a small polydisc $U_0\subset N$, and let $\pi:\mathcal{X}\to U_0$ be the marked universal curve. Let $\mathcal{D}\subset\mathcal{X}$ be the union of the marking sections, let $\mathcal{L}=\omega_{\mathcal{X}/U_0}^{\otimes 2}(\mathcal{D})$ and let $\mathcal{V}\to U_0$ be the holomorphic vector bundle with fiber $\mathcal{V}_X=Q_X N$.

After removing a proper analytic subset of $U_0$, the fixed divisor of the evaluation map $\pi^*\mathcal{V}\to\mathcal{L}$ has no vertical component, and removing it leaves a base-point-free linear system on every fiber. We therefore obtain a holomorphic map
\begin{equation} \Phi:\mathcal{X}\to\mathbb{P}_{U_0}(\mathcal{V}^\vee). \nonumber \end{equation}
Let $\mathcal{C}$ be its reduced image. By lemma~\ref{thm:flatness} lemma ~\ref{thm:reducedness} and lemma \ref{thm:simultaneous normalization}, after removing a further proper analytic subset of $U_0$ there is a holomorphic family of compact Riemann surfaces $\mathcal{Y}\to U_0$ and a finite map $\nu:\mathcal{Y}\to\mathcal{C}$ whose restriction to every fiber is the normalization of $C_t$.

By the universal property of normalization, $\Phi$ factors uniquely as $\Phi=\nu\circ H$ for a holomorphic map $H:\mathcal{X}\to\mathcal{Y}$. Since every fiber map $H_t:X_t\to Y_t$ is nonconstant, $H$ is finite. By Proposition~\ref{prop:ratio-quotient}, $H_t$ is precisely the quotient map $h_{X_t}:X_t\to Y_{X_t}$.

We now apply Lemma~\ref{lem:generic-splitting} to the ramification divisor of $H$ and to the pullback of its reduced branch divisor. After removing the coincidence loci, the branch values, their inverse images, and the marking points vary in disjoint holomorphic sections with constant local degrees. Lemma~\ref{lem:topological-triviality} then gives, locally over a sufficiently small polydisc $U\subset U_0$, smooth trivializations
\begin{equation} \mathcal{X}|_U\simeq U\times S_X,\qquad \mathcal{Y}|_U\simeq U\times S_Y, \nonumber \end{equation}
under which every $H_t$ is the same topological branched cover $h:S_X\to S_Y$ and the marking sections are constant.

Choose a connected topological Galois closure $r_0:S_Z\to S_Y$ of $h$, with deck group $G$, and a subgroup $G_X\leq G$ such that $S_Z/G_X\simeq S_X$. The marked family $\mathcal{Y}|_U\to U$ determines a holomorphic map to the corresponding Teichm\"uller space. Applying the covering construction associated with $r_0$ gives a holomorphic family $\mathcal{Z}\to U$ and a finite holomorphic map $R:\mathcal{Z}\to\mathcal{Y}|_U$. The deck transformations act smoothly on $\mathcal{Z}$, biholomorphically on every fiber, and commute with $R$; hence Lemma~\ref{lem:fiberwise-holomorphic} shows that the $G$-action on $\mathcal{Z}$ is holomorphic.

Set $\mathcal{X}'=\mathcal{Z}/G_X$. Locally, a stabilizer of order $k$ acts in a fiber coordinate by $z\mapsto\zeta z$, so the quotient is given by $w=z^k$. Thus $\mathcal{X}'\to U$ is again a holomorphic family of Riemann surfaces. Under the above topological trivializations, both $\mathcal{X}'\to\mathcal{Y}|_U$ and $\mathcal{X}|_U\to\mathcal{Y}|_U$ are the same branched cover $h$. Their fiberwise complex structures therefore agree, and Lemma~\ref{lem:fiberwise-holomorphic} gives $\mathcal{Z}/G_X\simeq\mathcal{X}|_U$.

Let $N^\circ\subset N$ be the set of points $X$ admitting a neighborhood $U$ on which the quotient maps $h_t:X_t\to Y_t$ form a holomorphic family that is locally topologically trivial as a family of branched coverings and their connected Galois closures form a holomorphic family with fixed deck group and fixed intermediate subgroup. By above argument, for every sufficiently small polydisc $U_0\subset N$, the complement of $N^\circ\cap U_0$ is contained in a proper analytic subset of $U_0$. Hence $N^\circ$ is open and dense in $N$.
\end{proof}

\bibliographystyle{amsplain}
\bibliography{references}

\end{document}